\documentclass[twoside,reqno]{amsart}
\usepackage[utf8]{inputenc}
\usepackage{amssymb}
\usepackage{amsfonts}
\usepackage{amsthm}
\usepackage{amsmath}
\usepackage{bbm}
\usepackage{fancyhdr}
\usepackage{enumerate}
\usepackage{enumitem}
\usepackage{tikz}
\usepackage{tikz-cd}
\usepackage{mathtools}
\usepackage{color, soul}

\usepackage[backend=biber,style=numeric,sorting=anyt, maxbibnames=99]{biblatex}
\DeclareNameAlias{sortname}{given-family}
\DeclareNameAlias{default}{given-family}

\usepackage{hyperref}
\hypersetup{colorlinks, linkcolor={blue}, citecolor={blue}, urlcolor={red}}

\begin{document}

\newtheorem{theorem}{Theorem}[section]
\newtheorem{lemma}[theorem]{Lemma}
\newtheorem{proposition}[theorem]{Proposition}
\newtheorem{corollary}[theorem]{Corollary}
\theoremstyle{definition}
\newtheorem{definition}[theorem]{Definition}
\newtheorem{example}[theorem]{Example}
\newtheorem{remark}[theorem]{Remark}
\newtheorem{exercise}[theorem]{Exercise}
\newtheorem{convention}[theorem]{Convention}

\numberwithin{equation}{section}

\pagestyle{fancy}
\fancyhead[OC]{Generalised Twisted Groupoids and their $C^*$-algebras}
\fancyhead[OL]{\thepage}
\fancyhead[OR]{}
\fancyhead[EC]{L. O. Clark, M. \'{O} Ceallaigh, H. Pham}
\fancyhead[ER]{\thepage}
\fancyhead[EL]{}
\fancyfoot{}
\renewcommand{\headrulewidth}{0pt}
\setlength\headheight{24pt}

\newcommand{\C}{\mathbb{C}}    
\newcommand{\Z}{\mathbb{Z}}    
\newcommand{\T}{\mathbb{T}}    

\newcommand{\Ccal}{\mathcal{C}}
\newcommand{\Bcal}{\mathcal{B}}
\newcommand{\sfrak}{\mathfrak{s}}

\newcommand{\Lp}{\mathit{L}}
\newcommand{\Ltwo}{\Lp^2}

\newcommand{\innerproduct}[2]{\left\langle #1\,\middle|\,#2\right\rangle}
\newcommand{\norm}[1]{\left\lVert#1\right\rVert}
\newcommand{\abs}[1]{\left|#1\right|}
\newcommand{\set}[1]{\left\{#1\right\}}  

\newcommand{\support}{\textnormal{supp}}
\newcommand{\id}{\text{id}}
\newcommand{\Ind}{\textnormal{Ind}}
\newcommand{\average}[1]{{#1}_G}

\newcommand{\xrightarrowdbl}[2][]{%
  \xrightarrow[#1]{#2}\mathrel{\mkern-14mu}\rightarrow
}

\newcommand{\hp}[2]{{\color{blue}[(HP) #1] #2}}

\title{\textbf{Generalised Twisted Groupoids and their $C^*$-algebras}}

\author{Lisa Orloff Clark}
\address{L. O. Clark: Victoria University of Wellington, Wellington, New Zealand}
\email{lisa.orloffclark@vuw.ac.nz}

\author{Michael \'O Ceallaigh}
\address{M. \'O Ceallaigh: Victoria University of Wellington, Wellington, New Zealand}
\email{michael.e.kelly@vuw.ac.nz}

\author{Hung Pham}
\address{H. Pham: Victoria University of Wellington, Wellington, New Zealand}
\email{hung.pham@vuw.ac.nz}
\date{}


\thanks{This research
was supported by Marsden grant  21-VUW-156 from the Royal Society of New Zealand.  The authors would also like to thank Becky Armstrong and Astrid an Huef for helpful conversations. 
}


\begin{abstract}
    We consider a locally compact Hausdorff groupoid $G$, and twist by a more general locally compact Hausdorff abelian group $\Gamma$ rather than the complex unit circle $\mathbb{T}$. We investigate the construction of $C^*$-algebras in analogue to the usual twisted groupoid $C^*$-algebras, and we show that, in fact, any $\Gamma$-twisted groupoid $C^*$-algebra is isomorphic to a usual twisted groupoid $C^*$-algebra.
\end{abstract}

\maketitle

\section{Introduction}
Groupoid $C^*$-algebras are a widely studied class of $C^*$-algebras that includes many important subclasses, for example, transformation group $C^*$-algebras \cite[Example~3.3]{CaH1},  graph and higher-rank graph $C^*$-algebras \cite{KPRR, Kumjian_1986} and inverse semigroup $C^*$-algebras \cite{Paterson}.
However,  Buss and Sims show in \cite{buss--sims_opposite} that there are in fact $C^*$-algebras that are not groupoid $C^*$-algebras. 
Yet, introduced by Renault in \cite{renault_thesis}, groupoid $C^*$-algebras arose as a special case of \textit{twisted} groupoid $C^*$-algebras, giving us a natural object on which to focus with the hope of modelling even more $C^*$-algebras. To this end, much progress has been made already. 
Kumjian demonstrates in \cite{Kumjian_1986} that diagonal pairs $(A, B)$ of $C^*$-algebras are characterised by twists over principal groupoids, and Renault demonstrates in \cite{Renault_IMSbulletin} that Cartan pairs $(A, B)$ of $C^*$-algebras are characterised by twists over effective groupoids.  More recently, the first-named author and co-authors show in \cite{BCLM} that $C^*$-algebras that contain semi-Cartan subalgebras are characterised by twists over any groupoid.
To date, there is no known example of a $C^*$-algebra that fails to be isomorphic to a twisted groupoid $C^*$-algebra.

It is remarkable that in many results concerning twisted groupoid $C^*$-algebras, given a twist $\Sigma$ over $G$ as in Def. \ref{def:gamma_twist} with $\Gamma = \mathbb{T}$, the key hypotheses are often placed on the underlying groupoid $G$ rather than the twist itself. This observation guides the focus of the present paper. We aim to further develop the idea that, from the perspective of $C^*$-algebra theory, the structural properties of the groupoid $G$ play a more central role than the specific nature of its twisting. While the twist certainly influences the resulting $C^*$-algebra, our results suggest that much of the essential behaviour is determined by $G$.

A twisted groupoid is defined, in its full generality, in terms of a group bundle with fibre set $\T$, the complex unit circle. In this paper we study a generalised construction, allowing any locally compact Hausdorff abelian group $\Gamma$ to be the fibre set of that bundle. We refer to the result as a $\Gamma$\textit{-twist}. This is dealt with also in \cite{pushouts}, with Example 3.5 exhibiting the generalised twist construction, and in \cite{IKRSW21, IKSW19}.  Although the authors of \cite{pushouts} call a $\Gamma$-twist a ``generalised twist'', we show that these new twists do not give any new $C^*$-algebras.

In the literature, when presenting the definition of a twist, usually the aforementioned group bundle is specified to be locally trivial. Otherwise this is omitted from the definition and instead the projection of the bundle into the base space is specified to be a continuous open surjection. As discussed e.g. in \cite[Remark 2.6]{Armstrong_2022}, local triviality for twists follows automatically from this latter condition. This is not always true for $\Gamma$-twists, so we will only assume our bundle comes with a continuous open surjection.

We begin with a preliminary survey of the structure of a $\Gamma$-twisted groupoid compared to a $\mathbb{T}$-twisted groupoid. We then move towards building a function algebra on the $\Gamma$-twisted groupoid; in Section \ref{section:Haar} we exhibit a natural Haar system construction for a $\Gamma$-twist over a groupoid with Haar system, and in Section \ref{section:staralgebra} we develop a suitable $*$-algebra in analogue with the $*$-algebra of equivariant functions in the $\mathbb{T}$-twisted groupoid case, and conclude that $\Gamma$ should in fact be compact for the purpose of this investigation. Then in Section \ref{section:c*-algebras} we move on to develop analogous definitions of the full and reduced $\Gamma$-twisted groupoid $C^*$-algebras.

We conclude in Section \ref{section:comparison} with our main result. Theorem \ref{thm: compatible twists have isomorphic algebras} exhibits a comparison between $\Gamma$- and $\Gamma^{\prime}$-twists over a groupoid $G$ in the presence of a continuous group homomorphism $\theta: \Gamma \to \Gamma^{\prime}$. Then, in the special case $\Gamma^{\prime} = \mathbb{T}$, meaning $\theta$ is a character of $\Gamma$, we find that the $\Gamma$-twisted groupoid $C^*$-algebras are isomorphic to $\mathbb{T}$-twisted groupoid $C^*$-algebras.

\section{Preliminaries}\label{section:preliminaries}

We begin by presenting the main object of this paper, the generalised twisted groupoid. This definition is adapted from the most general definition of a twist, as presented e.g. in \cite{Armstrong_2022, courtney2023alexandrov, simsgroupoidnotes, BCLM, Kumjian_1986, CaH2}. In the seminal work of Renault, twists were described in terms of a continuous $\mathbb{T}$-valued $2$-cocycle, in which case the twist is a topologically trivial bundle over $G$. Kumjian \cite{Kumjian_1986} introduced the more general notion of a twist, where that bundle may be only locally trivial, in order to extend the work of Renault to obtain an analogue with the Feldman--Moore theory for von Neumann algebras. There are indeed examples of twists which are topologically non-trivial, see e.g. \cite[\S 4]{Kumjian_1986}, \cite[Ex. 2.1]{Muhly_Williams_1992}, \cite{armstrong2023twistminimaletalegroupoid}.

\begin{definition}\label{def:gamma_twist}
    Let $\Gamma$ be a locally compact Hausdorff abelian group with identity element $e$, and let $G$ be a locally compact Hausdorff groupoid, regarding $G^{(0)} \times \Gamma$ as a trivial group bundle with fibres $\Gamma$. A $\Gamma$\textit{-twist} $(\Sigma, \iota, \pi)$ over $G$ is a locally compact Hausdorff groupoid $\Sigma$ with unit space $\Sigma^{(0)} = \iota(G^{(0)} \times \set{e})$ where we have the sequence    
    \[
    \begin{tikzcd}
        G^{(0)} \times \Gamma \arrow[r, "\iota"] & \Sigma \arrow[r, "\pi"] & G
    \end{tikzcd}
    \]
    and the following conditions are satisfied:
    \begin{enumerate}[label = \textnormal{(\roman*)}, left = \parindent]
        \item{The maps $\iota$, $\pi$ are continuous groupoid homomorphisms, $\iota$ a topological embedding and $\pi$ a quotient map, which restrict to homeomorphisms of unit spaces. We identify $\Sigma^{(0)} = G^{(0)}$ via $\pi |_{\Sigma^{(0)}}$;}
        \item{The sequence above is exact, meaning in addition to $\iota$ being injective and $\pi$ surjective, we have $\iota(G^{(0)} \times \Gamma) = \pi^{-1}(G^{(0)})$;}
        \item{The image of $\iota$ is central in $\Sigma$, meaning for each $\sigma \in \Sigma$ and each $\alpha \in \Gamma$ we have $\iota(r(\sigma), \alpha)\sigma = \sigma\iota(s(\sigma), \alpha)$.} \footnote{We note that this centrality condition forces the group $\Gamma$ to be abelian.}
    \end{enumerate}
\end{definition}

\begin{remark}
    We note that Def. \ref{def:gamma_twist} is equivalent to defining $\Sigma$ as a $\Gamma$-groupoid in analogue with \cite{BCLM, CaH2, Kumjian_1986}. That is, $\Sigma$ is a locally compact Hausdorff groupoid together with an action of the locally compact Hausdorff abelian group $\Gamma$ such that
    \begin{enumerate}[label = \textnormal{(\roman*)}, left = \parindent]
        \item{For $(\sigma, \tau) \in \Sigma^{(2)}$ and $\alpha, \beta \in \Gamma$ we have $(\alpha \cdot \sigma, \beta \cdot \tau) \in \Sigma^{(2)}$ and $(\alpha \cdot \sigma)(\beta \cdot \tau) = \alpha \beta \cdot \sigma \tau$; and}
        \item{$G := \Sigma/\Gamma$ is a locally compact Hausdorff groupoid.}
    \end{enumerate}
\end{remark}

We refer to a $\mathbb{T}$-twist, simply as a \textit{twist}. It will be useful to make note of the following relationship between the homomorphisms $\iota$ and $\pi$:

\begin{proposition}\label{beckys_lemma}
For each $u \in G^{(0)}$ and each $\alpha \in \Gamma$ we have
    \[
        \pi\left( \iota(u, \alpha) \right) = u.
    \]
    In particular, for any $\sigma \in \Sigma$ this gives $r(\sigma) = r(\pi(\sigma))$ and $s(\sigma) = s(\pi(\sigma))$.
\end{proposition}

\begin{proof}
Take $\sigma\in \Sigma$ and $\alpha\in \Gamma$. From (iii) of Definition \ref{def:gamma_twist}, we can see that $(\iota(r(\sigma),\alpha),\sigma)\in \Sigma^{(2)}$, and subsequently since $\pi$ is a groupoid homomorphism $(\pi(\iota(r(\sigma),\alpha)),\pi(\sigma))\in G^{(2)}$. This and Definition \ref{def:gamma_twist}(ii) imply that 
\[
    \pi(\iota(r(\sigma),\alpha))=r(\pi(\sigma)).
\]
Thus the left-hand side is independent of $\alpha\in \Gamma$, and so
\begin{align*}
    \pi(\iota(r(\sigma),\alpha))=\pi(\iota(r(\sigma),e))=r(\sigma),
\end{align*}
through the identifications of unit spaces, explained above. This proves
\[
    r(\pi(\sigma))=r(\sigma)\quad\text{and}\quad \pi\left( \iota(u, \alpha) \right) = u,
\]
where the latter is because every $u\in G^{(0)}$ is $r(\sigma)$ for some $\sigma$. The assertion that $s(\sigma) = s(\pi(\sigma))$ follows similarly.
\end{proof}

\begin{lemma}\label{lemma:groupaction}
Define an operation $\cdot$ from $\Gamma \times \Sigma$ to $\Sigma$ by
    $$ \alpha \cdot \sigma := \iota(r(\sigma), \alpha)\sigma = \sigma \iota(s(\sigma), \alpha) $$
    for each $\sigma \in \Sigma$ and $\alpha \in \Gamma$. Then $\cdot$ is a continuous free action of $\Gamma$ on $\Sigma$.
\end{lemma} 

\begin{proof}
Since, for $\alpha$, $\beta \in \Gamma$ and $\sigma \in \Sigma$,
$$e \cdot \sigma = \iota \left(r(\sigma), e \right)\sigma = r(\sigma)\sigma = \sigma$$
and
\begin{equation*}
\begin{split}
    \beta \cdot (\alpha \cdot \sigma) &= \beta \cdot \left[ \iota \left(r(\sigma), \alpha \right)\sigma \right] \\
    &= \iota [\underbrace{r \left(\iota(r(\sigma), \alpha)\sigma \right)}_{= r(\sigma)}, \beta ] \iota \left[r(\sigma), \alpha \right] \sigma \\
    &= \iota \left(r(\sigma), \beta \alpha \right) \sigma.
\end{split}
\end{equation*}
It follows that $\cdot$ is a group action. The action is free by injectivity of $\iota$ -- for $\alpha \in \Gamma$, $\sigma \in \Sigma$ we have
\begin{equation*}
\begin{split}
    \alpha \cdot \sigma &= \sigma = e \cdot \sigma\\
    \iff \iota (r(\sigma), \alpha)\sigma &= \sigma = \iota (r(\sigma), e)\sigma\\
    \implies \iota(r(\sigma), \alpha) &= \sigma\sigma^{-1} = \iota(r(\sigma), e) \\
    \implies (r(\sigma), \alpha) &= (r(\sigma), e) \\
    \implies \alpha &= e.
\end{split}
\end{equation*}
Finally it is continuous by construction, as it is a composition of continuous maps.
\end{proof}

This action of $\Gamma$ on $\Sigma$ also satisfies the following useful property:

\begin{lemma}\label{lemma:uniqueaction}
    Given $\sigma, \tau \in \Sigma$ such that $\pi(\sigma) = \pi(\tau)$, there exists a unique element $\alpha \in \Gamma$ such that $\tau = \alpha \cdot \sigma$.
\end{lemma}

\begin{proof}
If $\pi(\sigma) = \pi(\tau)$ then 
$$ s(\tau) = s(\pi(\tau)) = s(\pi(\sigma)) = s(\sigma) = r(\sigma^{-1}), $$
and we have $\pi(\tau)\pi(\sigma)^{-1} = \pi(\tau\sigma^{-1}) \in G^{(0)}$, so that $\tau\sigma^{-1} \in \pi^{-1}\left(G^{(0)} \right)$. By Definition \ref{def:gamma_twist}(ii), there is a unique element $(u, \alpha) \in G^{(0)} \times \Gamma$ such that $\iota(u, \alpha) = \tau\sigma^{-1}$.
Thus $\tau=\iota(u,\alpha)\sigma$, and by Proposition \ref{beckys_lemma}, $u=s(u)=s(\iota(u,\alpha))$, and so $u=r(\sigma)$. Hence, $\tau= \iota(r(\sigma), \alpha)\sigma = \alpha \cdot \sigma$.
\end{proof}

\begin{corollary}
    The mapping $\pi: \Sigma \to G$ is open.
\end{corollary}

In the most common setting, a $\Gamma$-twist $\Sigma$ is a \textit{locally trivial} bundle over $G$, i.e. $\pi^{-1}(U) \cong U \times \Gamma$ for any open subset $U \subseteq G$. For example, in \cite[\S 3]{etaleautomaticsections} it is shown that twists over \'etale groupoids always admit continuous local sections. Results of Gleason \cite{autosections-gleason} and Palais \cite{autosections-palais} (attributed to Serre) show that if $\Gamma$ is a compact Lie group and $G$ is completely regular, then the $\Gamma$-twist admits continuous local sections. Note that locally compact Hausdorff spaces are automatically completely regular by e.g. \cite[\S 33 Ex. 7]{munkres}.

\begin{remark}
The existence of continuous local sections is sufficient to make $\Sigma$ locally trivial; if  $P:U\to  \Sigma$ is a continuous section of $\pi:\Sigma\to G$ over some subset $U$ of $G$, then the mapping
  \[
    \varphi_P :U\times \Gamma\to \pi^{-1}(U),\  (\eta,\alpha)\mapsto \alpha\cdot P(\eta)\,,
  \]
is a homeomorphism.
\end{remark}

However, there are indeed $\Gamma$-twists for which no continuous local sections exist. The following example is exhibited in \cite[Hooptedoodle 4.68]{raeburn-moritaequivalence}:

\begin{example}
    \[
    \begin{tikzcd}
       \{\mathbbm{1}\} \times \prod_{1}^{\infty} \{\pm 1\} && \prod_{1}^{\infty} \T && \prod_1^{\infty} \T
        \arrow[from=1-1, to=1-3, "\iota"]
        \arrow[from=1-3, to=1-5, "\pi"]
    \end{tikzcd}
    \]
    Here all of our groupoids are groups, with $\Sigma=\prod_{1}^{\infty} \T$ realised as a twist by the product of two-element groups $\Gamma = \prod_{1}^{\infty} \{\pm 1\}$ over $G = \Sigma$. We denote by $\mathbbm{1}$ the identity of $\Sigma$. We have $\iota: (\mathbbm{1}, z) \mapsto z$ and $\pi: (z_n)_{n \in \mathbb{N}} \mapsto (z_n^2)_{n \in \mathbb{N}}$. Now, an open set in $\prod_1^{\infty} \T$ must contain a basic open set $(\prod_{n=1}^N U_n) \times (\prod_{N+1}^{\infty}\T)$; however, a continuous section of $\pi$ defined on this set would yield a continuous branch of the square root function on a copy of $\T$ in the product space, which we know does not exist.
\end{example}

\begin{remark}
    Note that if we specified that $\Sigma$ be locally trivial then $\pi$ is automatically an open mapping, by an argument similar to the proof of \cite[Lemma 2.7(a)]{Armstrong_2022}. In that case Definition \ref{def:gamma_twist} is automatically satisfied, and the results demonstrated in this paper still hold.
\end{remark}

\begin{proposition}\label{prop:action_homeo}
Fix an element $\gamma \in G$. Then for each $\sigma \in \pi^{-1}(\gamma)$, the mapping
$$ \rho_{\sigma}: \Gamma \to \pi^{-1}(\gamma),\ \alpha \mapsto \alpha \cdot \sigma $$
is a homeomorphism.
\end{proposition}

\begin{proof}
The continuity and injectivity of $\rho_{\sigma}$ follow from the fact that the action of $\Gamma$ on $\Sigma$ is continuous and free, respectively, by Lemma \ref{lemma:groupaction}. Lemma \ref{lemma:uniqueaction} tells us precisely that $\rho_{\sigma}$ is surjective. Lastly, note that the inverse mapping $\alpha \cdot \sigma = \iota(r(\sigma), \alpha)\sigma \mapsto \alpha$ is a composition of continuous operations: multiplication by $\sigma^{-1}$, followed by the inverse of $\iota$ defined on $\iota(G^{(0)} \times \Gamma)$, and projection onto the second coordinate.
\end{proof}

If furthermore $\Gamma$ is compact, as is the case with a usual twist with $\Gamma = \T$ (see \cite[Lemma~2.2]{courtney2023alexandrov}), then the map $\pi$ is proper.

\begin{proposition}\label{pro: pi is proper}
Suppose the group $\Gamma$ is compact. Then $\pi$ is a proper map.
\end{proposition}

\begin{proof}
    Take any compact subset $K \subseteq G$, and an open cover $\{U_i \mid i \in I\}$ of $\pi^{-1}(K)$. For each $\gamma \in K$, by Proposition \ref{prop:action_homeo} the fibre $\pi^{-1}(\gamma)$ is compact, and so there is a finite set of indices $I_{\gamma}$ such that $\{U_{i} \mid i \in I_{\gamma}\}$ covers $\pi^{-1}(\gamma)$. Set $V_{\gamma} := \bigcap_{i \in I_{\gamma}} U_i$. Then since $\pi$ is an open mapping, $\{\pi(V_{\gamma}) \mid \gamma \in K\}$ is an open cover of $\pi(\pi^{-1}(K)) = K$. By compactness of $K$ there is a finite subcovering $\{\pi(V_i) \mid 1 \leq i \leq n\}$. But now $K \subseteq \bigcup_{i=1}^n \pi(V_i)$ implies $\pi^{-1}(K) \subseteq \pi^{-1}\left(\bigcup_{i=1}^n \pi(V_i)\right) = \bigcup_{i=1}^n \pi^{-1}(\pi(V_i))$. By continuity, these preimages are open and hence yield the desired finite subcovering.
\end{proof}

\begin{convention}\label{convention, global: gamma twist}
From now on, we shall let $\Gamma$ be a locally compact Hausdorff abelian group and let $(\Sigma, \iota, \pi)$ be a $\Gamma$-twist over some locally compact Hausdorff groupoid $G$ as in Definition \ref{def:gamma_twist}, whose $\Gamma$-action $(\alpha,\sigma)\mapsto \alpha\cdot \sigma$ is defined as in Lemma \ref{lemma:groupaction}. In fact, although the material of the next section holds for all locally compact $\Gamma$, at the level of $*$-algebras we will see that it suffices to consider only compact $\Gamma$.
\end{convention}

\section{Haar Systems}\label{section:Haar}

For operations in the $C^*$-algebras which we associate to a $\Gamma$-twist, it will be necessary to integrate functions on $\Sigma$; in the case of usual twists the existence of a Haar system is ensured (as long as one exists for the underlying groupoid $G$) by a well-known result which is proved in \cite[Lemma 2.4]{courtney2023alexandrov}. We develop an analogous result for $\Gamma$-twists. Moving forward we will assume there exists a Haar system $\set{\lambda^u}_{u \in G^{(0)}}$ for $G$. If $G$ is an \'{e}tale groupoid, meaning the range and source maps are local homeomorphisms, then we can always use the Haar system consisting of counting measures; for proof see \cite[Prop. 1.29]{toolkit}. We will denote by $\lambda$ the  Haar measure on $\Gamma$.

Let $\sfrak: G \to \Sigma$ be a (not necessarily continuous) section of $\pi$. For each $f\in \Ccal_c(\Sigma)$, define a function $\average{f}: G \to \C$ by 
\[
    \average{f}(\gamma) = \int_{\Gamma} f(\alpha \cdot \sfrak(\gamma)) d\lambda(\alpha).
\]

\begin{lemma}\label{lem: for thm:haarsystemobtained}
For each $f\in \Ccal_c(\Sigma)$, the function $\average{f}$ is well-defined, independent of the choice of $\sfrak: G \to \Sigma$, and belongs to $\Ccal_c(G)$. 
\end{lemma}
\begin{proof}
First, take $\gamma\in G$ and $\sigma\in \pi^{-1}(\gamma)$. Then there is some $\alpha_{\eta} \in \Gamma$ with $\sfrak(\gamma) = \alpha_{\eta} \cdot \sigma$. Hence
    \begin{equation*}
    \begin{split}
        \average{f}(\gamma) &:= \int_{\Gamma} f(\alpha \cdot \sfrak(\gamma)) d\lambda(\alpha) \\
        &= \int_{\Gamma} f(\alpha \alpha_{\eta} \cdot \sigma) d\lambda(\alpha)= \int_{\Gamma} f(\beta \cdot \sigma) d\lambda(\beta),
    \end{split}
    \end{equation*}
 because $\lambda$ is invariant. This shows that the definition of $\average{f}(\gamma)$ does not depend on the particular choice of the value of $\sfrak(\gamma)$ in $\pi^{-1}(\gamma)$.

For open $U \subseteq \mathbb{C}$ we have
\begin{equation*}
\begin{split}
    \pi^{-1}(f_G^{-1}(U)) &= \{\sigma \in \Sigma \mid \int_{\Gamma} f(\alpha \cdot \sfrak(\pi(\sigma))) d\lambda(\alpha) \in U \} \\
    &= \{\sigma \in \Sigma \mid \int_{\Gamma} f(\alpha \cdot \sigma)d\lambda(\alpha) \in U\},
\end{split}
\end{equation*}
which is open, as the preimage of $U$ under the continuous mapping $F: \Sigma \to \mathbb{C},\ \sigma \mapsto \int_{\Gamma} f(\alpha \cdot \sigma) d\lambda(\alpha)$. To see that this map $F$ is continuous, let $\varepsilon>0$, fix $\sigma_0 \in \Sigma$ and choose a precompact open neighbourhood $U$ of $\sigma_0$. First note that
$$ K:= \{\alpha \in \Gamma \mid \alpha \cdot \sigma \in \textnormal{supp}f\textnormal{ for some }\sigma \in \overline{U}\} $$
is compact; for any net $(\alpha_i)_{i \in I}$ in $K$ there is a corresponding net $(\alpha_i \cdot \sigma_i)_{i \in I}$ in $\textnormal{supp}f$. Then since $\textnormal{supp}f$ is compact we can pass to a subnet (and re-label) to get $\alpha_i \cdot \sigma_i \to \tau$ for some $\tau \in \textnormal{supp}f$. Now $(\sigma_i)_{i \in I}$ is a net in $\overline{U}$ so we can pass again to a subnet, re-label and assume $\sigma_i \to \sigma$ for some $\sigma \in \overline{U}$. It follows that $\sigma_i^{-1} \to \sigma^{-1}$ and hence $\iota(r(\sigma_i), \alpha_i) = \alpha_i \cdot \sigma_i \sigma_i^{-1} \to \tau\sigma^{-1}$. Since $\iota$ is a homeomorphism onto its image that implies the net $(r(\sigma_i), \alpha_i)_{i \in I}$ is convergent in $G^{(0)} \times \Gamma$, and hence $(\alpha_i)_{i \in I}$ converges in $\Gamma$.

Now, since $K$ is compact we have $\lambda(K) < \infty$, and since $f$ is compactly supported there is an open neighbourhood $V$ of $\sigma_0$ such that 
$$ \sigma \in V \implies |f(\alpha \cdot \sigma) - f(\alpha \cdot \sigma_0)| < \frac{\varepsilon}{\lambda(K)+1} $$
for every $\alpha \in K$. It follows that
    $$ \left| \int_{\Gamma} f(\alpha \cdot \sigma) d\lambda(\alpha) - \int_{\Gamma} f(\alpha \cdot \sigma_0) d\lambda(\alpha) \right| \leq \int_K \left|f(\alpha \cdot \sigma) - f(\alpha \cdot \sigma_0) \right| d\lambda(\alpha) < \varepsilon.$$

To conclude that $f_G$ is continuous, note that since $\pi$ is surjective and open, the set $\pi\left( \pi^{-1}(f_G^{-1}(U))\right) = f_G^{-1}(U)$ is open.

Finally, if $\average{f}(\gamma) \neq 0$ for some $\gamma \in G$ then there exists $\alpha \in \Gamma$ with $\alpha \cdot \sfrak(\gamma) \in \textnormal{supp }f$, which implies $\pi(\alpha \cdot\sfrak(\gamma)) = \gamma \in \pi(\textnormal{supp }f)$. Thus supp $\average{f} \subseteq \pi(\textnormal{supp }f)$; but $f \in \Ccal_c(\Sigma)$ and $\pi$ is continuous, therefore $\pi(\textnormal{supp }f)$, and hence also supp $\average{f}$, is compact. 
\end{proof}

Next, fix a unit $u \in G^{(0)}$ and define $I^u: \Ccal_c(\Sigma) \to \C$ by
\[
I^u(f) =\int_{G^u}\average{f}(\gamma) d\lambda^u(\gamma)= \int_{G^u} \int_{\Gamma} f(\alpha \cdot \sfrak(\gamma)) d\lambda(\alpha) d\lambda^u(\gamma).  \]
Then, by the previous lemma, $I^u$ is a well-defined positive linear functional on $\Ccal_c(\Sigma)$. Riesz' Representation Theorem (\cite[Thm. 7.2.8]{cohn_measure}) then ensures for each $u \in \Sigma^{(0)}$ the existence of a unique Radon measure $\mu^u$ on $\Sigma$ which satisfies
\begin{equation}\label{eqn:measuredef}
\int_{\Sigma} f(\sigma) d\mu^u(\sigma) = \int_{G^u} \int_{\Gamma} f(\alpha \cdot \sfrak(\gamma)) d\lambda(\alpha) d\lambda^u(\gamma).
\end{equation}

\begin{theorem}\label{thm:haarsystemobtained}
    Let $\Gamma$ be a locally compact Hausdorff abelian group and let $(\Sigma, \iota, \pi)$ be a $\Gamma$-twist over a locally compact Hausdorff groupoid $G$ with Haar system $\{\lambda^u\}_{u \in G^{(0)}}$. Then the collection $\set{\mu^u}_{u \in \Sigma^{(0)}}$ as defined above is a Haar system for $\Sigma$.
\end{theorem}

\begin{proof}
    First, for every function $f \in \Ccal_c(\Sigma)$, since $\average{f}\in \Ccal_c(G)$  and $\set{\lambda^u}_{u \in G^{(0)}}$ is a Haar system on $G$, by definition the mapping $G^{(0)}  \to \C$,
    \[
        u \mapsto \int_G \average{f}(\gamma) d\lambda^u(\gamma)
    \]
    is continuous, which becomes the mapping
    \[
        \Sigma^{(0)} \to \C, \ \ u \mapsto \int_{\Sigma} f(\sigma) d\mu^u(\sigma)\,
    \]
    under the identification $\Sigma^{(0)}\equiv G^{(0)}$.

    Next, let $\sigma \in \Sigma^u$ and let $f\in \Ccal_c(\Sigma)$ with $f\ge 0$ on $\Sigma$ and $f(\sigma)>0$. Then, by Lemma \ref{lem: for thm:haarsystemobtained},
    \[
        \average{f}(\pi(\sigma))= \int_{\Gamma} f(\alpha \cdot \sigma) d\lambda(\alpha)>0
    \]
    and so since $\pi(\sigma)\in G^{u}=\support\lambda^u$, we have
    \[
        I^u(f)=\int_{G^{u}} \average{f}(\gamma) d\lambda^u(\gamma)>0\,.
    \]
    This shows that $\Sigma^u\subseteq \support(f)$. To show the reverse inclusion, notice that if $f\in \Ccal_c(\Sigma)$ and $f|_{\Sigma^u}=0$, then for every $\gamma\in G^u$, we have
    \[
        \average{f}(\gamma)=\int_{\Gamma} f(\alpha \cdot \sfrak(\gamma)) d\lambda(\alpha)=0
    \]
    and so $I^u(f)=0$.

    Finally, to prove the left-invariance of the measures, for $\sigma \in \Sigma$ and $f \in \Ccal_c(\Sigma)$ we need to show
    $$ \int f(\sigma\tau) d\mu^{s(\sigma)}(\tau) = \int f(\tau) d\mu^{r(\sigma)}(\tau). $$
    To see this, note that
    \begin{align*}
    \int f(\sigma\tau) d\mu^{s(\sigma)}(\tau) &= \int_{G^{s(\sigma)}} \int_{\Gamma} f(\sigma(\alpha \cdot \sfrak(\gamma))) d\lambda(\alpha) d\lambda^{s(\sigma)}(\gamma) \\
    &= \int_{G^{s(\sigma)}} \int_{\Gamma} f(\alpha \cdot \sigma\sfrak(\gamma)) d\lambda(\alpha)d\lambda^{s(\sigma)}(\gamma)\\
    &=\int_{G^{s(\sigma)}} \average{f}(
    \pi(\sigma)\gamma) d\lambda^{s(\sigma)}(\gamma)\\
     &=\int_{G^{r(\sigma)}} \average{f}(
    \gamma) d\lambda^{r(\sigma)}(\gamma)\\
    &= \int f(\tau) d\mu^{r(\sigma)}(\tau). \qedhere
    \end{align*}
\end{proof}

\begin{convention}\label{convention, global: haar systems}
In addition to Convention \ref{convention, global: gamma twist}, from now on, we suppose that our locally compact Hausdorff groupoid $G$ admits a Haar system that is denoted by $\set{\lambda^u}_{u \in G^{(0)}}$, denote a fixed Haar measure on $\Gamma$ by $\lambda$ (which is normalised when $\Gamma$ is compact), and denote the induced Haar system on $\Sigma$ as defined above by $\set{\mu^u}$. Note that the Haar system $\set{\lambda^u}_{u \in G^{(0)}}$ for $G$ is not guaranteed to exist unless $G$ is second countable and its range map $r$ is open (for details see e.g. \cite{deitmar2017haar}). Also denote by $\sfrak:G\to \Sigma$ a (not necessarily continuous) section of $\pi:\Sigma\to G$.
\end{convention}

\section{The \texorpdfstring{$*$}{*}-algebra \texorpdfstring{$\Ccal_c(\Sigma; G; \chi)$}{Cc(Sigma; G; chi)}}\label{section:staralgebra}

We now wish to explore what kind of $C^*$-algebra can be associated to a $\Gamma$-twist. In the case of a usual twist $\Sigma$ over a groupoid $G$, one constructs the \textit{twisted groupoid }$C^*$\textit{-algebra} $C^*(\Sigma; G)$ by completing the space
$$ \Ccal_c(\Sigma; G) := \set{ f \in \Ccal_c(\Sigma) \mid f(z \cdot \sigma) = zf(\sigma) \text{ for all } \sigma \in \Sigma \text{ and } z \in \T }. $$

But of course, with $\Ccal_c(\Sigma)$ being a space of mappings $\Sigma \to \C$, for $\alpha \in \Gamma$ it only makes sense to write $\alpha f(\sigma)$ whenever $\Gamma \subseteq \C$. One way to move forward is to embed $\Gamma$ in $\C$; this can be done via a \textit{character}, i.e. a continuous homomorphism $\Gamma\to\T$. See for example \cite[Ch. 3--4]{abstract_harmonic} for more details. 
Note that the set $\widehat{\Gamma}$ of characters of $\Gamma$, with the pointwise product and the topology of uniform convergence on compacts, is again a locally compact Hausdorff abelian group, and the Pontryagin duality states that there is a canonical isomorphism of topological groups $\Gamma \cong \widehat{\widehat{\Gamma}}$ (see e.g. \cite{pontryagin_duality}).

In particular, unless $\Gamma$ is trivial, it always has (plenty of) non-trivial characters; when $\chi$ is the trivial character, we will just recover the standard construction of $\Ccal_c(G)$, the groupoid algebra.

Now, let $\chi$ be a fixed character on $\Gamma$ and define
\[
\Ccal_c(\Sigma; G; \chi) := \set{ f \in \Ccal_c(\Sigma) \mid f(\alpha \cdot \sigma) = \chi(\alpha) f(\sigma)  \text{ for all } \sigma \in \Sigma  \text{ and } \alpha \in \Gamma }.
\]

This \textit{equivariance} property, as in the usual $\T$-twist case, tells us that evaluation of a function at one point is enough to determine the function's values at all points in the same fibre. In other words, it can be viewed as a function on $G$. This aligns with the equivalent description of a twisted groupoid $C^*$-algebra in terms of sections of a particular line bundle over $G$, see e.g., \cite{Kumjian_1986,Renault_IMSbulletin,BFPR_graded}.

\begin{example}
Consider the $\Z_2$-twist $(\Z_4, \iota, \pi)$ over $\Z_2$, where each space is endowed with the discrete topology and the homomorphisms are
\begin{equation*}
\begin{split}
    &\iota: ([0]_2, [k]_2) \mapsto [2k]_4 \\
    &\pi: [k]_4 \mapsto [k]_2. 
\end{split}
\end{equation*}
Then $\Z_2$ acts on $\Z_4$ as follows:
\begin{equation*}
\begin{split}
    [0]_2 \cdot [k]_4 &= [k]_4 \\
    [1]_2 \cdot [k]_4 &= [k+2]_4.
\end{split}
\end{equation*}

The mapping $\chi: \Z_2 \to \T,\ [k]_2 \mapsto e^{k\pi i}$ is a character of $\Z_2$ -- in this case, since we are dealing with finite discrete spaces, $\support f$ is always compact for $f: \Z_4 \to \C$, which itself is always continuous. Thus the elements of $\Ccal_c(\Z_4; \Z_2; \chi)$ are functions $f: \Z_4 \to \C$ satisfying
\begin{equation*}
\begin{split}
    f([0]_2 \cdot [k]_4) = f([k]_4) &\textnormal{ and } f([1]_2 \cdot [k]_4) = -f([k]_4) \\
    \iff f([k+2]_4) &= -f([k]_4),
\end{split}
\end{equation*}
that is, 
$$ \Ccal_c(\Z_4; \Z_2; \chi) = \set{ f \in \C^{\Z_4} \mid f([k+2]_4) = -f([k]_4) \text{ for all } k \in \Z}. $$
To see that the space is non-trivial, notice $[k]_4 \mapsto e^{\frac{k\pi i}{2}}$ is an element.
\end{example}

\begin{example}
    If $(\Sigma, \iota, \pi)$ is a twist over a groupoid $G$ then $\id_{\T}$ is a character and gives the usual space:
    $$ \Ccal_c(\Sigma; G; \id_{\T}) = \Ccal_c(\Sigma; G). $$
\end{example}

\begin{remark}
Notice that $\Ccal_c(\Sigma; G; \chi)$ contains non-zero functions only if $\Gamma$ is compact. To see this take any $f \in \Ccal_c(\Sigma; G; \chi)$ such that $f(\sigma) \neq 0$ for some $\sigma \in \Sigma$. Then for any $\tau \in \pi^{-1}\left( \pi(\sigma) \right)$, by Lemma \ref{lemma:uniqueaction} there is a unique $\alpha \in \Gamma$ such that $\tau = \alpha \cdot \sigma$ and we have
    $$ f(\tau) = f(\alpha \cdot \sigma) = \chi(\alpha) f(\sigma) \neq 0 $$
    so that $\pi^{-1}\left( \pi(\sigma) \right) \subseteq \support f$. Since $f$ is compactly supported, and Proposition \ref{prop:action_homeo} implies $\pi^{-1}\left( \pi(\sigma) \right)$ is closed, it follows that $\pi^{-1}\left( \pi(\sigma) \right)$ is compact.
\end{remark}

Suppose for the remainder of this section that $\Gamma$ is compact (and that, as customary in this case, its Haar measure $\lambda$ is normalised i.e. $\lambda(\Gamma) = 1$).  

\begin{remark}
    Calling back to Theorem \ref{thm:haarsystemobtained} we make note of a particular nice property for functions of this type. If $f \in \Ccal_c(\Sigma; G; \chi)$ and the character $\chi$ is non-trivial, then
    \begin{equation*}
    \begin{split}
        \int_{\Sigma} f(\sigma) d\mu^u(\sigma) &:= \int_{G^u} \int_{\Gamma} f(\alpha \cdot \sfrak(\gamma)) d\lambda(\alpha) d\lambda^u(\gamma) \\
        &= \int_{G^u} \int_{\Gamma} \chi(\alpha) f(\mathfrak{s(\gamma)}) d\lambda(\alpha) d\lambda^u(\gamma) \\
        &= \underbrace{\int_{\Gamma} \chi(\alpha) d\lambda(\alpha)}_{=0} \int_{G^u} f(\sfrak(\gamma)) d\lambda^u(\gamma) = 0.
    \end{split}
    \end{equation*}
\end{remark}

Finally, we move closer to building our $C^*$-algebras. Recall now that the function space $\Ccal_c(\Sigma)$, together with the convolution
    \[ f*g(\sigma) := \int_{\Sigma} f(\tau) g(\tau^{-1} \sigma) d\mu^{r(\sigma)}(\tau)
    \]
and involution
    \[
    f^*(\sigma) := \overline{{f(\sigma^{-1})}}
    \]
forms a $*$-algebra. This is proved e.g. in \cite[Prop. II.1.1]{renault_thesis}.

\begin{proposition}\label{prop: algebra of equivariant Cc}
    The space $\Ccal_c(\Sigma; G; \chi)$ is a $*$-ideal of $\Ccal_c(\Sigma)$.
\end{proposition}
\begin{proof}
It is straightforward to see that $\Ccal_c(\Sigma; G; \chi)$ is a subspace of $\Ccal_c(\Sigma)$. Take $f \in \Ccal_c(\Sigma; G; \chi)$, $g \in \Ccal_c(\Sigma)$, $\alpha \in \Gamma$ and $\sigma \in \Sigma$. Then
\[f*g(\alpha \cdot \sigma) = \int_{\Sigma} f(\tau) g(\tau^{-1}(\alpha \cdot \sigma)) d\mu^{r(\alpha \cdot \sigma)}(\tau) =\int_{\Sigma} f(\tau) g((\alpha^{-1}\cdot\tau)^{-1}\sigma) d\mu^{r(\alpha \cdot \sigma)}(\tau).
\]
Set $\nu:=\iota(r(\sigma),\alpha^{-1})$, and substitute $\eta:=\alpha^{-1}\cdot \tau =\nu\tau$, to obtain
\begin{align*}
        f*g(\alpha \cdot \sigma) &=\int_{\Sigma} f(\tau) g((\alpha^{-1}\cdot\tau)^{-1}\sigma) d\mu^{s(\nu)}(\tau)\\
        &= \int_{\Sigma} f( \alpha \cdot \eta) g(\eta^{-1}\sigma) d\mu^{r(\nu)}(\eta) \\
        &= \int_{\Sigma} \chi(\alpha)f(\eta) g(\eta^{-1}\sigma) d\mu^{r(\sigma)}(\eta) \\
        &= \chi(\alpha)f*g(\sigma).
\end{align*}
Thus $f*g \in \Ccal_c(\Sigma; G; \chi)$, showing that $\Ccal_c(\Sigma; G; \chi)$ is a right-ideal of $\Ccal_c(\Sigma)$ -- a similar argument shows it is also a left-ideal.
\end{proof}

Convolution of functions in $\Ccal_c(\Sigma;G;\chi)$ has the following formula:

\begin{lemma}\label{lem: equivariant convolution}
For every $f,g\in \Ccal_c(\Sigma; G; \chi)$ and every section $\mathfrak{s}$ of $\pi$ we have
\[
    f*g(\sigma) =\int_G f(\sfrak(\gamma)) g(\sfrak(\gamma)^{-1}\sigma) d\lambda^{r(\sigma)}(\gamma)\qquad(\sigma\in\Sigma)\,.
\]
\end{lemma}

\begin{proof}
We see that 
\begin{align*}
        f*g(\sigma) &=\int_{\Sigma} f(\tau) g(\tau^{-1}\sigma) d\mu^{r(\sigma)}(\tau)\\
        &= \int_G\int_\Gamma f(\alpha \cdot \sfrak(\gamma)) g(\alpha^{-1}\cdot \sfrak(\gamma)^{-1}\sigma) d\lambda(\alpha) d\lambda^{r(\sigma)}(\gamma) 
\end{align*}
and so the desired formula follows from the equivariance of $f$ and $g$ and the normalisation of $\lambda$.
\end{proof}

\begin{example}
    It is worth checking that $\Ccal_c(\Sigma; G; \chi)$ is not the entirety of $\Ccal_c(\Sigma)$. Returning to the example of $\Z_4$ as a $\Z_2$-twist over $\Z_2$ with character $[k]_2 \mapsto e^{k\pi i}$, we have, since $\Z_4$ is discrete here,
    $$ \textnormal{dim}\Ccal_c(\Z_4) = \textnormal{dim}\C^{\Z_4} = |\Z_4| \textnormal{dim}\C = 4, $$
    yet
    \[\Ccal_c(\Z_4; \Z_2; e^{k\pi i}) = \set{f \in \C^{\Z_4} \mid \text{ for all } k \in \Z, f([k+2]_4) = -f([k]_4) }
    \]
    consists of those functions of the form $([0]_4, [1]_4, [2]_4, [3]_4) \mapsto (t, s, -t, -s) $ for parameters $t,s \in \C$, hence its dimension is 2 and the two spaces are distinct. Being finite-dimensional, the spaces are complete and hence we have $\Ccal^*(\Z_4) \neq \Ccal^*(\Z_4; \Z_2; e^{k \pi i})$.
\end{example}

\section{\texorpdfstring{$C^*$}{C*}-Algebras}\label{section:c*-algebras}

Suppose throughout this section that $\Gamma$ is compact and that $\lambda$ is normalised, and let $\chi$ be a character on $\Gamma$. Consider the set $\textnormal{Rep}(\Ccal_c(\Sigma; G; \chi))$ of representations $L: \Ccal_c(\Sigma; G; \chi) \to \Bcal(\mathcal{H})$ which are \emph{$I$-norm bounded}, i.e.
$$\norm{L(f)} \leq \norm{f}_I  \text{ for all } f \in \Ccal_c(\Sigma; G; \chi) $$
where $\mathcal{H}$ is a Hilbert space,
\begin{equation*}
\begin{split}
   \norm{f}_{I, s} &= \sup_{u \in \Sigma^{(0)}} \int_{\Sigma} |f(\sigma)| d\mu_u(\sigma),\\
    \norm{f}_{I, r} &= \sup_{u \in \Sigma^{(0)}} \int_{\Sigma} |f(\sigma)| d\mu^u(\sigma),
\end{split}
\end{equation*}
and
\[ \norm{f}_I = \textnormal{max}\set{ \norm{f}_{I, s}, \norm{f}_{I, r} }.
\]
Recall that for a given $u \in \Sigma^{(0)}$, the Radon measure $\mu_u$ is induced from $\mu^u$ of the Haar system on $\Sigma$ by the relation
\[
    \int_\Sigma f(\sigma) d\mu_u(\sigma)=\int_\Sigma f(\sigma^{-1}) d\mu^u(\sigma)\qquad \text{ for } f\in \Ccal_c(\Sigma).
\]
In particular, the support of $\mu_u$ is $\Sigma_u$.

\begin{definition}
The \textit{full} norm on $\Ccal_c(\Sigma; G; \chi)$ is given by
$$ \norm{f} := \textnormal{sup}\set{ \norm{L(f)}_{\Bcal(\mathcal{H})} \mid L \in \textnormal{Rep}(\Ccal_c(\Sigma; G; \chi)) }. $$
The completion $\Ccal^*(\Sigma; G; \chi)$ of $\Ccal_c(\Sigma; G; \chi)$ in the full norm, we will call the \textit{full }$\Gamma$\textit{-twisted groupoid }$C^*$\textit{-algebra} associated to the $\Gamma$-twist ($\Sigma$, $\iota$, $\pi$) and the character $\chi$.
\end{definition}

For later use, we make the following simple observation:

\begin{lemma}\label{lem: equivariant Lone norm}
For every $f\in \Ccal_c(\Sigma,G;\chi)$, we have
\[
    \int_{\Sigma} \abs{f(\sigma)} d\mu_u(\sigma)=\int_G\abs{f(\sfrak(\gamma))} d\lambda_u(\gamma),
\]
and
\[
\int_{\Sigma} \abs{f(\sigma)} d\mu^u(\sigma)=\int_G\abs{f(\sfrak(\gamma))} d\lambda^u(\gamma),
\]
where $\lambda_u$ is the analogue of $\mu_u$ for $\set{\lambda^u}_{u \in G^{(0)}}$ on $G$.
\end{lemma}
\begin{proof}
The second identity follows from the equivariance of $f$ and the definition of $\mu^u$ (and $\lambda(\Gamma)=1$), while the first follows from the second using the section $\gamma\mapsto \sfrak(\gamma^{-1})^{-1}$ instead of $\sfrak$.
\end{proof}

Alternatively, instead of considering all $I$-norm bounded representations, we can reduce to considering only a specific type of representation as follows. For a given $u \in \Sigma^{(0)}$, the \textit{point mass} measure which assigns a value of 1 to a set containing $u$ and 0 otherwise, gives rise to an $I$-norm bounded representation (for details see \cite[Prop. 1.41]{toolkit}) $\Ind  \delta_u: \Ccal_c(\Sigma) \to \Bcal\left(\Ltwo(\Sigma_u; \mu_u)\right)$, defined by extending continuously the formula
\[
    \Ind \delta_u (f)(g)(\sigma) := f*g(\sigma)=\int_{\Sigma} f(\tau) g(\tau^{-1} \sigma) d\mu^{r(\sigma)}(\tau)
\]
for $f, g \in \Ccal_c(\Sigma)$ and $\sigma \in \Sigma_u$. Notice that the $\Ltwo$-space is with respect to $\mu_u$ instead of $\mu^u$. This restricts to an $I$-norm bounded representation $\Ccal_c(\Sigma; G; \chi) \to \Bcal\left( \Ltwo(\Sigma_u,\mu_u)\right)$ by Proposition \ref{prop: algebra of equivariant Cc}.
However, it will be more convenient for us to be able to shrink the Hilbert space $\Ltwo(\Sigma_u,\mu_u)$ further as follows. For each $\alpha\in\Gamma$ and each $f\in \Ccal_c(\Sigma)$, define $T_\alpha(f)$ to be the function $\sigma\mapsto f(\alpha\cdot \sigma)$ in $\Ccal_c(\Sigma)$. We see that
\begin{align*}
    \int_\Sigma \abs{f(\alpha\cdot \sigma)}^2 d\mu_u(\sigma)&=\int_\Sigma \abs{f(\alpha\cdot \sigma^{-1})}^2 d\mu^u(\sigma)\\
    &=\int_G\int_\Gamma \abs{f(\alpha\cdot (\beta\cdot \sfrak(\gamma))^{-1})}^2 d\lambda(\beta) d\lambda^u(\gamma)\\
    &=\int_G\int_\Gamma \abs{f((\beta\alpha^{-1})^{-1}\cdot \sfrak(\gamma)^{-1})}^2 d\lambda(\beta) d\lambda^u(\gamma)\\
    &=\int_G\int_\Gamma \abs{f(\beta^{-1}\cdot \sfrak(\gamma)^{-1})}^2 d\lambda(\beta) d\lambda^u(\gamma)\\
    &=\int_{\Sigma} \abs{f(\sigma^{-1})}d\mu^u(\sigma) \\
    &=\int_\Sigma \abs{f(\sigma)}^2 d\mu_u(\sigma).
\end{align*}
Thus $T_\alpha$ extends to a unitary operator, also denoted by $T_\alpha$, on $\Ltwo(\Sigma_u,\mu_u)$; the inverse of $T_\alpha$ is $T_{\alpha^{-1}}$. Now, define for each character $\chi$ of $\Gamma$
\[
    \Ltwo(\Sigma_u; G_u; \chi) := \set{g \in \Ltwo(\Sigma_u, \mu_u) \mid\  T_\alpha(g)=\chi(\alpha)g\ (\alpha\in\Gamma)}.
\]

\begin{lemma}
For $u \in \Sigma^{(0)}$ we have
$$\Ind\delta_u(\Ccal_c(\Sigma; G; \chi))(\Ltwo(\Sigma_u,\mu_u))\subseteq \Ltwo(\Sigma_u; G_u; \chi),$$
so that each $\Ind\delta_u(\Ccal_c(\Sigma; G; \chi))$ is zero on $\Ltwo(\Sigma_u; G_u; \chi)^\perp$.
\end{lemma}

\begin{proof}
The first assertion follows from the density of $\Ccal_c(\Sigma)$ in $\Ltwo(\Sigma_u,\mu_u)$ because, for each $f\in \Ccal_c(\Sigma;G;\chi)$ and $g\in\Ccal_c(\Sigma)$, we have $\Ind\delta_u(f)(g)=f*g$ which is equivariant by Proposition \ref{prop: algebra of equivariant Cc}. The second assertion follows from the first since $\Ind\delta_u$ is a $*$-representation of the $*$-algebra $\Ccal_c(\Sigma;G;\chi)$ on the Hilbert space $\Ltwo(\Sigma_u,\mu_u)$.
\end{proof}

\begin{definition}
Define the \textit{reduced norm} on $\Ccal_c(\Sigma; G; \chi)$ by
\[
    \norm{f}_r := \sup_{u \in \Sigma^{(0)}}  \norm{\Ind \delta_u(f)}_{\Bcal(\Ltwo(\Sigma_u; G_u; \chi))}= \sup_{u \in \Sigma^{(0)}}  \norm{\Ind \delta_u(f)}_{\Bcal(\Ltwo(\Sigma_u,\mu_u))},
\]
and the completion $\Ccal^*_r(\Sigma; G; \chi)$ of $\Ccal_c(\Sigma; G; \chi)$ in this norm, we will refer to as the \textit{reduced }$\Gamma$\textit{-twisted groupoid }$C^*$\textit{-algebra} associated to the $\Gamma$-twist ($\Sigma$, $\iota$, $\pi$) and the character $\chi$.
\end{definition}

\begin{lemma}\label{lem: strong continuity of Talpha}
The map $\alpha\mapsto T_\alpha$ from  $\Gamma$ to $\Bcal(\Ltwo(\Sigma_u,\mu_u))$ is strongly continuous.
\end{lemma}
\begin{proof}
Since each $T_\alpha$ is a unitary and $\Ccal_c(\Sigma)$ is dense in $\Ltwo(\Sigma_u;\mu_u)$, it is sufficient to prove that for a given $f\in \Ccal_c(\Sigma)$, the function $\alpha\mapsto T_\alpha(f),\ \Gamma\to \Ltwo(\Sigma_u,\mu_u),$ is continuous. For this, set $K:=\pi^{-1}(\pi(\support f))$ and define $j_K:\Ccal(K)\to \Ltwo(\Sigma_u,\mu_u)$ to be the mapping that extends every function in $\Ccal(K)$ to be zero outside of $K$. Since $K$ is compact, $j_K$ is a well-defined bounded linear operator. Note that the formula for $T_\alpha$ also defines a bounded linear operator on $\Ccal(K)$, denoted by $T'_\alpha$. Then $T_\alpha(f)=j_K(T'_\alpha(f))$. Since $\Ccal(\Gamma\times K)\equiv \Ccal(\Gamma,\Ccal(K))$ canonically and isometrically, the function $\alpha\mapsto T'_\alpha(f),\ \Gamma\to \Ccal(K),$ is continuous. The continuity of $\alpha\mapsto T_\alpha(f),\ \Gamma\to \Ltwo(\Sigma_u,\mu_u),$ then follows. 
\end{proof}

\begin{lemma}\label{lem: Ltwo density of equivariant Cc}
The space $\Ccal_c(\Sigma;G;\chi)$ is dense in $\Ltwo(\Sigma_u;G_u;\chi)$. 
\end{lemma}
\begin{proof}
Take $h\in \Ltwo(\Sigma_u;G_u;\chi)$, and let $f_n\in \Ccal_c(\Sigma)$ such that $f_n\to h$ in $\Ltwo(\Sigma_u,\mu_u)$ as $n \to \infty$. Then, for each $\alpha\in \Gamma$, since
\begin{equation*}
\begin{split}
    \norm{\chi(\alpha)^{-1}T_\alpha(f_n)-h}_{\Ltwo(\Sigma_u,\mu_u)}&=\norm{\chi(\alpha)^{-1}T_\alpha(f_n)-\chi(\alpha)^{-1}T_\alpha(h)}_{\Ltwo(\Sigma_u,\mu_u)}\\
    &=\norm{f_n-h}_{\Ltwo(\Sigma_u,\mu_u)}
\end{split}
\end{equation*}
we see that 
\[
    g_n:=\int_\Gamma \chi(\alpha)^{-1}T_\alpha(f_n) d\lambda(\alpha)\to h\qquad\text{in $\Ltwo(\Sigma_u,\mu_u)$ as $n\to\infty$},
\]
where each integral is convergent in the norm of $\Ltwo(\Sigma_u,\mu_u)$ thanks to the continuity of the integrand. It remains only to prove for each $f\in \Ccal_c(\Sigma)$, that the element $g:=\int_\Gamma \chi(\alpha)^{-1}T_\alpha(f) d\lambda(\alpha)$ of $\Ltwo(\Sigma_u,\mu_u)$ actually belongs to $\Ccal_c(\Sigma;G;\chi)$. Indeed, keeping the same notation as in the proof of Lemma \ref{lem: strong continuity of Talpha}, we see that $g=j_K(g')$ where
\[
    g':=\int_\Gamma \chi(\alpha)^{-1}T'_\alpha(f) d\lambda(\alpha)
\]
and the integral is convergent in the norm of $\Ccal(K)$. Thus $g'$ is the restriction of the function
\[
    \sigma\mapsto \int_\Gamma \chi(\alpha)^{-1}f(\alpha\cdot\sigma) d\lambda(\alpha),\ \Sigma\to\C\,,
\]
which is readily seen to belong to $\Ccal_c(\Sigma;G;\chi)$.
\end{proof}



\bigskip

\section{Comparing generalised twisted groupoids with twisted ones}\label{section:comparison}

In this section, we shall show that $\Gamma$-twisted groupoid $C^*$-algebras can be realised as twisted groupoid $C^*$-algebras. This will be achieved by first demonstrating that, given a character $\chi$ on $\Gamma$, there is a natural morphism from $(\Sigma,\iota,\pi)$ into a $\T$-twist. 
Let $(\Sigma, \iota, \pi)$ be a $\Gamma$-twist over a locally compact Hausdorff groupoid $G$ as before. The first thing to notice is that if $\chi$ is a character on $\Gamma$, then the equivariance condition on $\Ccal_c(\Sigma,G;\chi)$ makes the points in each orbit $(\ker\chi)\cdot \sigma$ indistinguishable, for $\sigma \in \Sigma$. In fact, this idea works more generally:

\begin{lemma}\label{lem: quotient of twisted groupoid}
Let $\Lambda$ be a closed subgroup of $\Gamma$. For $\sigma,\tau\in\Sigma$, say that 
\[
    \sigma \sim \tau \iff \text{ there exists } \alpha \in \Lambda \text{ such that }  \sigma = \alpha \cdot \tau\,.
\]
Denote by $q:\Gamma\to \Gamma/\Lambda\,,\ \alpha\mapsto [\alpha],$ and $q_1:\Sigma\to\Sigma/{\sim}\,,\ \sigma\mapsto [\sigma],$ the quotient maps. Then:
\begin{enumerate}[label = \textnormal{(\roman*)}, left = \parindent]
    \item The relation $\sim$ defined above is an equivalence relation on $\Sigma$. 
    \item\label{it2:equivrel} The set $\Sigma/{\sim}$ has a unique groupoid structure such that the quotient map $q_1:\Sigma\to \Sigma/{\sim}$ a groupoid homomorphism. 
    \item\label{it3:equivrel} There are unique maps $\hat{\iota}: G^{(0)} \times (\Gamma/\Lambda) \to \Sigma$ and $\hat{\pi}: \Sigma/{\sim} \to G$ such that the following diagram commutes:
    \[\begin{tikzcd}
	   {G^{(0)} \times \Gamma} && \Sigma \\
       &&&& G\\
	   {G^{(0)} \times (\Gamma/\Lambda)} && {\Sigma/{\sim}}
	   \arrow["\iota", from=1-1, to=1-3]
	   \arrow["q_0:=\id\times q"', from=1-1, to=3-1]
	   \arrow["\pi", from=1-3, to=2-5]
	   \arrow["{q_1}"', from=1-3, to=3-3]
	   \arrow["{\hat{\iota}}"', from=3-1, to=3-3]
	   \arrow["{\hat{\pi}}"', from=3-3, to=2-5]
    \end{tikzcd}\]
    Moreover, $\hat{\iota}$ is an injective groupoid homomorphism, while $\hat{\pi}$ is a surjective groupoid homomorphism.
    \item\label{it4:equivrel} Equip $\Gamma/\Lambda$ and $\Sigma/{\sim}$ with the quotient topologies induced by these surjective mappings $q:\Gamma\to\Gamma/\Lambda$ and $q_1:\Sigma\to \Sigma/{\sim}$. Then $\Gamma/\Lambda$ is a locally compact Hausdorff abelian group,  $\Sigma/{\sim}$ is a locally compact Hausdorff topological groupoid, and  $(\Sigma/{\sim}, \hat{\iota}, \hat{\pi})$ is a $\Gamma/\Lambda$-twist over $G$.
\end{enumerate}
\end{lemma}

\begin{proof}
It is straightforward to check that $\sim$ is an equivalence relation on $\Sigma$. For item~\ref{it2:equivrel}, the only way to define a product on $\Sigma/{\sim}$ such that $q_1$ is a groupoid homomorphism is to set 
\[
    [\sigma][\tau]:=[\sigma\tau]\quad\text{whenever}\quad (\sigma,\tau)\in \Sigma^{(2)}.
\]
This operation is well-defined because if $\sigma,\tau$ are as above and we have $\sigma'=\alpha\cdot\sigma$, $\tau'=\beta\cdot\tau$ for some $\alpha,\beta\in \Lambda$,  then $(\sigma',\tau')\in\Sigma^{(2)}$ and $\sigma'\tau'=\alpha\beta\cdot (\sigma\tau)$. With this multiplication, it is again straightforward to check that $\Sigma/{\sim}$ is a groupoid and $q_1:\Sigma\to\Sigma/{\sim}$ is a groupoid homomorphism. 

Item~\ref{it3:equivrel} is also straightforward: the well-definedness and the injectivity of $\hat{\iota}$ follow from the definition of $\Sigma/{\sim}$ and the injectivity of $\iota$. The well-definedness and the surjectivity of $\hat{\pi}$ follow from the definition of $\Sigma/{\sim}$ and the surjectivity of $\pi$. The uniqueness of $\hat{\iota}$ and of $\hat{\pi}$ follows from the surjectivity of $\id \times q$ and $q_1$, respectively; while their homomorphism property follows also from the  surjectivity of $\id \times q$ and $q_1$ and from the homomorphism property of $\id \times q,\iota,q_1,\pi$.

Now we show item~\ref{it4:equivrel}. That $\Gamma/\Lambda$, equipped with the quotient topology, is a locally compact Hausdorff abelian group is standard. In fact, the continuous map $q:\Gamma\to \Gamma/\Lambda$ is open, and therefore so too is $\id \times q:G^{(0)}\times \Gamma\to G^{(0)}\times\Gamma/\Lambda$.  Thus, the product topology on $G^{(0)}\times\Gamma/\Lambda$ is also the quotient topology induced by the surjective map $\id \times q$, and $\hat{\iota}: G^{(0)} \times (\Gamma/\Lambda) \to \Sigma/{\sim}$ is continuous.

By the definition of the quotient topology on $\Sigma/{\sim}$, the map $\hat{\pi}:\Sigma/{\sim}\to G$ is also continuous. Further, the surjective map $q_1:\Sigma\to\Sigma/{\sim}$ is also open because  if $U$ is open in $\Sigma$ then
\[
    q_1^{-1}(q_1(U))=\set{\alpha\cdot \sigma\colon\ \alpha\in \Lambda,\ \sigma\in U}
\]
which is open in $\Sigma$ by Lemma \ref{lemma:groupaction}. In particular, this implies that $\Sigma/{\sim}$ is locally compact, and together with the openness of $\pi$ it follows that $\hat{\pi}$ is open. Also, we have $q_1^{-1}(\textnormal{im }\iota) = \textnormal{im }\hat{\iota}$, and so $q_1 |_{\textnormal{im }\iota}: \textnormal{im }\iota \to \textnormal{im }\hat{\iota}$ is an open mapping; the commutative diagram then shows that $\hat{\iota}$ is a topological embedding.

Since $\Lambda$ is compact and $\Sigma$ is Hausdorff, it follows e.g. by \cite[\S 31 Ex. 8]{munkres} that $\Sigma/{\sim}$ is also Hausdorff.

Finally we show that $\Sigma/{\sim}$ is a topological groupoid, that is, the operations are continuous. Consider the following commutative diagram:
\[\begin{tikzcd}
	   \Sigma\times \Sigma && \Sigma^{(2)} &&\Sigma \\
       &&&&&& \\
	   (\Sigma/{\sim}) \times (\Sigma/{\sim}) &&  (\Sigma/{\sim})^{(2)}&&\Sigma/{\sim}.
	   \arrow["\text{product}", from=1-3, to=1-5]
        \arrow["\text{inclusion}", from=1-3, to=1-1]
	   \arrow["q_1\times q_1"', from=1-1, to=3-1]
	   \arrow[from=1-3, to=3-3]
       \arrow["{q_1}", from=1-5, to=3-5]
	   \arrow["\text{product}"', from=3-3, to=3-5]
       \arrow["\text{inclusion}", from=3-3, to=3-1]
    \end{tikzcd}
\]
Since $q_1:\Sigma\to\Sigma/{\sim}$ is a surjective continuous open map, so too is $q_1\times q_1:\Sigma\times \Sigma\to (\Sigma/{\sim}) \times (\Sigma/{\sim})$. We see from the definition of the product on $\Sigma/{\sim}$ that
\[
    (q_1\times q_1)^{-1}\left((\Sigma/{\sim})^{(2)}\right)=\Sigma^{(2)}
\]
and so the restriction of $q_1\times q_1$ to $\Sigma^{(2)}\to (\Sigma/{\sim})^{(2)}$ is also a surjective continuous open map. In particular, the subspace topology on $(\Sigma/{\sim})^{(2)}$ is the quotient topology induced from the subspace topology on $\Sigma^{(2)}$ by the restriction of  $q_1\times q_1$. In this way, we obtain the continuity of the multiplication $(\Sigma/{\sim})^{(2)}\to \Sigma/{\sim}$ from that of $\Sigma^{(2)}\to \Sigma$. The continuity of the inversion operation on $\Sigma/{\sim}$ follows by similar reasoning. The final assertion that $(\Sigma/{\sim}, \hat{\iota}, \hat{\pi})$ is a $\Gamma/\Lambda$-twist over $G$ now follows.
\end{proof}

Consider now the general situation:

\begin{theorem}\label{thm: compatible twists construction}
Let $(\Sigma,\iota, \pi)$ be a $\Gamma$-twist over $G$ and let $\theta:\Gamma\to\Gamma'$ be a continuous group homomorphism from $\Gamma$ into another locally compact Hausdorff abelian group $\Gamma'$. Then there exists a $\Gamma'$-twist $(\Sigma',\iota',\pi')$ and a continuous groupoid homomorphism $\theta_1:\Sigma\to \Sigma'$ such that the following diagram commutes:
\begin{equation}\label{diagram: twist extension}
\begin{tikzcd}
	{G^{(0)} \times \Gamma} && \Sigma \\
	&&&& G. \\
	{G^{(0)} \times \Gamma'} && {\Sigma'}
	\arrow["\iota", hook, from=1-1, to=1-3]
	\arrow["\id\times\theta"',  from=1-1, to=3-1]
	\arrow["\pi", two heads, from=1-3, to=2-5]
	\arrow["\theta_1", from=1-3, to=3-3]
	\arrow["{\iota'}", hook, from=3-1, to=3-3]
	\arrow["{\pi'}", two heads, from=3-3, to=2-5]
\end{tikzcd}
\end{equation}
Moreover, if $\theta$ is a topological embedding of $\Gamma$ as a closed subgroup of $\Gamma'$, then $\theta_1$ is a topological embedding of $\Sigma$ as a closed subgroupoid of $\Sigma'$.
\end{theorem}

To prove this theorem, notice that the most important aspect of $\Gamma$ is its action on $\Sigma$, so to begin we define $\Sigma'$ by way of extending the action of $\Gamma$ on $\Sigma$ to an action defined of all of $\Gamma'$ on $\Sigma$ in a formal way. In the special case when $\theta$ is an embedding of $\Gamma$ as a closed subgroup of $\Gamma'$, we envision $\Sigma$ as subgroupoid of the newly constructed $\Sigma'$ whose $\Gamma'$-action is an extension of the $\Gamma$-action on $\Sigma$, and so we need a rule telling when $z\cdot\sigma=w\cdot \tau$ for $w, z \in \Gamma'$ and $\sigma,\tau \in \Sigma$.  To this end, define a relation $\sim$ on $\Gamma' \times \Sigma$ such that 
\[
    (z, \sigma) \sim (w, \tau) \textnormal{ if and only if } \frac{w}{z}=\theta(c)\textnormal{ and } \sigma = c \cdot \tau\ \text{for some $c\in\Gamma$}.
\]
Then $\sim$ is an equivalence relation.
Set $\Sigma' := (\Gamma' \times \Sigma)/{\sim}$ equipped with the quotient topology. We define the multiplication and inverse operations coordinate-wise:
\begin{align*}
    [z, \sigma][w, \tau] := [zw, \sigma \tau]\quad \textnormal { and} \quad
    [z, \sigma]^{-1} := \left[z^{-1}, \sigma^{-1}\right]\textnormal{,}
\end{align*}
where $z, w \in \Gamma'$ and $ \sigma, \tau \in \Sigma$. Towards our aim of constructing a twist, we will also define
\begin{align*} 
\iota'&: G^{(0)} \times \Gamma' \to \Sigma',\ (u, z) \mapsto [z,\iota(u,e)]\equiv [z, u] \text{ and }\\
\pi'&: \Sigma' \to G,\ [z, \sigma] \mapsto \pi(\sigma).
\end{align*}
Denote by $e$ and $e'$ the units of $\Gamma$ and $\Gamma'$, respectively. 
The well-definedness and required properties of the above construction are proved in the following lemma.

\begin{lemma}
    The space $\Sigma'$ with operations defined above is a locally compact Hausdorff topological groupoid, and $(\Sigma^{\prime}, \iota^{\prime}, \pi^{\prime})$ is a twist over $G$.
\end{lemma}

\begin{proof}
It is straightforward to check that $\Gamma'\times \Sigma$, with coordinate-wise operations and the  product topology, is a locally compact Hausdorff groupoid with 
\begin{align*}
    s(z, \sigma) = (e', s(\sigma))\quad\textnormal{and} \quad r(z, \sigma) = (e', r(\sigma)),
\end{align*}
so $(\Gamma'\times\Sigma)^{(0)}$ is homeomorphic to $G^{(0)}$ and $\Sigma^{(0)}$ so that we can once again identify units. Moreover, together with the morphisms
\begin{align*}
    G^{(0)}\times (\Gamma'\times\Gamma)\xhookrightarrow[\iota^\times]{(u,z,\alpha)\mapsto (z,\iota(u,\alpha))} \Gamma'\times \Sigma \xrightarrowdbl[\pi^\times]{(z,\sigma)\mapsto \pi(\sigma)} G,
\end{align*}
$\Gamma'\times \Sigma$ becomes a $(\Gamma'\times\Gamma)$-twisted groupoid over $G$.

Now, we see that the equivalence relation $\sim$ on $\Gamma'\times\Sigma$ above is precisely the equivalence relation defined at the beginning of this section, where $\Gamma$ is replaced by $\Gamma'\times \Gamma$ and 
\[
    \Lambda:=\ker q\quad\text{where}\quad q:(z,\alpha)\mapsto z\theta(\alpha),\  \Gamma'\times\Gamma\to \Gamma'.
\]
Note that $q$ is a (continuous) group homomorphism because $\Gamma'$ is abelian so that $\Lambda$ is a (closed) subgroup of $\Gamma'\times \Gamma$. In fact, it is clear that $q$ is surjective and open, and so $\Gamma'$ is isomorphic to $(\Gamma'\times\Gamma)/\Lambda$ as topological groups. Notice that the following diagram commutes
\[\begin{tikzcd}
	   {G^{(0)} \times (\Gamma'\times\Gamma)} && \Gamma'\times\Sigma \\
       &&&& G\\
	   {G^{(0)} \times \Gamma'} && {(\Gamma'\times\Sigma)/{\sim}}
	   \arrow["\iota^\times", from=1-1, to=1-3]
	   \arrow["\id\times q"', from=1-1, to=3-1]
	   \arrow["\pi^\times", from=1-3, to=2-5]
	   \arrow["\text{equiv. projection}"', from=1-3, to=3-3]
	   \arrow["{\iota'}"', from=3-1, to=3-3]
	   \arrow["{\pi'}"', from=3-3, to=2-5]
    \end{tikzcd}\]
where the commutativity of the square comes from
\[
    [z,\iota(u,\alpha)]=[z\alpha,\iota(u,e)]\,.
\]
Thus, by Lemma \ref{lem: quotient of twisted groupoid}, we obtain that $(\Sigma',\iota',\pi')$ as defined above is a twist over $G$. 
\end{proof}

\begin{proof}[Proof of Theorem \ref{thm: compatible twists construction}]
The previous lemma provides most of the proof. Define
    \[
        \theta_1:\sigma\mapsto [e',\sigma],\ \Sigma\to \Sigma'.
    \]
Then $\theta_1$ is a continuous groupoid homomorphism that makes the diagram \eqref{diagram: twist extension} commutative.
In the case that $\theta$ is a topological embedding of $\Gamma$ as a closed subgroup of $\Gamma'$, that the corresponding property holds for $\theta_1:\Sigma\to\Sigma'$ follows from the commutative diagram \eqref{diagram: twist extension}.
\end{proof}


Finally, we arrive at the proposed comparison of the $\Gamma$-twisted and $\T$-twisted groupoid $C^*$-algebras with the following:

\begin{theorem}\label{thm: compatible twists have isomorphic algebras}
Let $(\Sigma',\iota',\pi')$ be a $\Gamma'$-twist over $G$ where $\Gamma'$ is a compact Hausdorff abelian group. Suppose that $\Gamma$ is also compact, that $\theta:\Gamma\to\Gamma'$ is a continuous group homomorphism, and that $\theta_1:\Sigma\to \Sigma'$ is a continuous groupoid homomorphism such that the diagram \eqref{diagram: twist extension} commutes.
Let $\chi'$ be a character on $\Gamma'$ and set $\chi:=\chi'\circ \theta$, a character on $\Gamma$. Then for each $f \in \Ccal_c(\Sigma; G; \chi)$, there is a unique $f'\in \Ccal_c(\Sigma'; G; \chi')$ such that the following diagram commutes:
\begin{equation}\label{thm: compatible twists have isomorphic algebras: diagram 2}
    \begin{tikzcd}
	\Sigma & {} & {\C}. \\
	\\
	{\Sigma'}
	\arrow["f", from=1-1, to=1-3]
	\arrow["{\theta_1}"', from=1-1, to=3-1]
	\arrow["{f'}"', from=3-1, to=1-3]
    \end{tikzcd}
\end{equation}
The mapping $f\mapsto f'$ defines a $*$-isomorphism $\Upsilon: \Ccal_c(\Sigma; G; \chi) \to \Ccal_c(\Sigma'; G; \chi')$, which extends to a $*$-isomorphism of the associated full and reduced $C^*$-algebras.
\end{theorem}

\begin{proof}
Let $f \in \Ccal_c(\Sigma; G; \chi)$. We claim that, given $\alpha',\beta'\in \Gamma'$ and $\sigma,\tau\in \Sigma$,
\[
    \alpha' \cdot \theta_1(\sigma)=\beta'\cdot \theta_1(\tau) \implies \chi'(\alpha')f(\sigma)=\chi'(\beta')f(\tau).
\]
 Indeed, the assumption and the commutativity of the triangle in \eqref{diagram: twist extension} imply that $\pi(\sigma)=\pi(\tau)$ and so there exists $c\in\Sigma$ such that $\tau=c\cdot \sigma$. We then see from the assumption and the commutativity of the square in \eqref{diagram: twist extension} that
\[  
    \alpha' \cdot \theta_1(\sigma)=\beta'\cdot \theta_1(\iota(r(\sigma),c))\theta_1(\sigma)=\beta'\theta(c)\cdot \theta_1(\sigma)
\]
and so $\alpha'=\beta'\theta(c)$. This implies,  using the equivariance of $f$, that 
\[
    \chi'(\alpha')f(\sigma)=\chi'(\beta')\chi(c)f(\sigma)=\chi'(\beta')f(\tau)
\]
as claimed. It follows from Lemma \ref{lemma:uniqueaction}, the commutativity of the triangle in \eqref{diagram: twist extension}, and this previous claim, that the formula
\[
    f'(\alpha'\cdot \theta_1(\sigma)):=\chi'(\alpha')f(\sigma)\quad (\alpha'\in\Gamma', \sigma\in\Sigma)
\]
well-defines the unique function $f':\Sigma'\to \C$ satisfying $f'(\alpha'\cdot \sigma')=\chi'(\alpha')f'(\sigma')$ for all $\alpha'\in \Gamma'$ and $\sigma'\in \Sigma'$. Let us now prove that $f^{\prime}$ is continuous; take any convergent net of the form $\alpha_i^{\prime} \cdot \theta_1(\sigma_i) \to \alpha^{\prime} \cdot \theta_1(\sigma)$, and let $U$ be any open neighbourhood of $f^{\prime}(\alpha^{\prime} \cdot \theta_1(\sigma))$. Suppose towards a contradiction that for every $i \in I$ there exists $i_0 \in I$ such that $i_0 \geq i$ and $f^{\prime}(\alpha^{\prime}_{i_0} \cdot \theta_1(\sigma_{i_0})) \notin U$. We pass to a subnet, discarding all the $i_0$ and re-labelling if necessary, to get a net $\{\alpha_i^{\prime} \cdot \theta_1(\sigma_i)\}$ which still converges to $\alpha^{\prime} \cdot \theta_1(\sigma)$, such that $f^{\prime}(\alpha^{\prime}_i \cdot \theta_1(\sigma_i))$ is never in $U$. By compactness of $\Gamma^{\prime}$, we can pass again to a subnet of $\{\alpha_i^{\prime}\}$ and re-label to obtain a net $\{\alpha^{\prime}_i\}$ converging to some $\alpha_0^{\prime} \in \Gamma^{\prime}$. It follows that
$$ \theta_1(\sigma_i) = (\alpha_i^{\prime})^{-1}\cdot(\alpha^{\prime}_i \cdot \theta_1(\sigma_i)) \to (\alpha^{\prime}_0)^{-1} \cdot (\alpha^{\prime} \cdot \theta_1(\sigma)). $$
By Proposition \ref{pro: pi is proper}, $\pi$ is a proper map. Since also $\pi^{\prime}$ is continuous, commutativity of the triangle in (\ref{diagram: twist extension}) demonstrates that $\theta_1$ is a proper map.
Therefore, letting $K$ be any compact neighbourhood of the limit above, $\theta_1^{-1}(K)$ is compact; certainly $\{\sigma_i\}$ is eventually in this set, so we can once again pass to a subnet, re-label, and get a net $\{\sigma_i\}$ converging to some $\sigma_0 \in \Sigma$. Uniqueness of limits ensures $\theta_1(\sigma_0) = (\alpha_0^{\prime})^{-1} \cdot (\alpha^{\prime} \cdot \theta_1(\sigma))$ and so, by continuity,
$$ f^{\prime}(\alpha_i^{\prime} \cdot \theta_1(\sigma_i)) = \chi^{\prime}(\alpha_i^{\prime})f(\sigma_i) \to \chi^{\prime}(\alpha_0^{\prime})f(\sigma_0) = f^{\prime}(\alpha^{\prime} \cdot \theta_1(\sigma)). $$
In particular, $f^{\prime}(\alpha^{\prime}_i \cdot \theta_1(\sigma_i))$ is eventually in $U$; but we already showed this is never in $U$. We conclude that $f^{\prime}(\alpha^{\prime}_i \cdot \theta_1(\sigma_i)) \to f^{\prime}(\alpha^{\prime} \cdot \theta_1(\sigma))$ and $f^{\prime}$ is continuous. Since $\support f'\subseteq \pi'^{-1}(\pi(\support f))$, we have shown that $f'\in \Ccal_c(\Sigma',G;\chi')$. Thus we obtain a well-defined map $\Upsilon: \Ccal_c(\Sigma; G; \chi) \to \Ccal_c(\Sigma'; G; \chi')$ which is a bijection; its inverse is the map $f'\mapsto f'\circ \theta_1$, using Proposition \ref{pro: pi is proper} again.

Next we show that $\Upsilon$ preserves $*$-algebra operations. Denote by $\set{{\mu'}^u}_{u \in \Sigma^{(0)}}$ the Haar system on $\Sigma'$ induced from $\set{\lambda^u}_{u \in G^{(0)}}$ in a similar manner as the Haar system $\set{\mu^u}_{u \in \Sigma^{(0)}}$ on $\Sigma$. Note that $\sfrak':=\theta_1\circ \sfrak$ defines a (not necessarily continuous) section of $\pi':\Sigma'\to G$. Taking $f,g \in \Ccal_c(\Sigma; G; \chi)$, we see from Lemma \ref{lem: equivariant convolution} that for all $\sigma\in\Sigma$,
    \begin{align*}
        (f' * g')(\theta_1(\sigma)) &= \int_Gf'((\theta_1\circ\sfrak)(\gamma))g'((\theta_1\circ\sfrak)(\gamma)^{-1}\theta_1(\sigma)) \ d\lambda^{r(\theta_1(\sigma))}(\gamma) \\
        &= \int_G f(\sfrak(\gamma))g(\sfrak(\gamma)^{-1}\sigma)\ d\lambda^{r(\sigma)}(\gamma)= f*g(\sigma)
    \end{align*}
and so $\Upsilon f* \Upsilon g=\Upsilon (f*g)$. Also,
    \begin{align*}
    (\Upsilon f)^*(\theta_1(\sigma)) = \overline{f'(\theta_1(\sigma)^{-1})} = \overline{f'(\theta_1(\sigma^{-1}))} &= \overline{f(\sigma^{-1})} = f^*(\sigma), 
    \end{align*}
so $(\Upsilon f)^* = \Upsilon(f^*)$. We have now shown that $\Upsilon$ is an algebraic $*$-isomorphism; it remains only to extend it to the $C^*$-algebra level. 

First we will deal with the full $C^*$-algebras. For this,  we see from Lemma \ref{lem: equivariant Lone norm} that for any $f \in \Ccal_c(\Sigma; G; \chi)$,
    \begin{align*}
         \norm{\Upsilon f}_{I, r} 
         &= \sup_{u \in (\Sigma')^{(0)}\equiv G^{(0)}} \int_G\abs{f' ((\theta_1\circ\sfrak)(\gamma))} \ d\lambda^u(\gamma) \\
          &= \sup_{u \in G^{(0)}\equiv \Sigma^{(0)}} \int_G\abs{f(\sfrak(\gamma))} \ d\lambda^u(\gamma) = \norm{f}_{I, r},
    \end{align*}
    and similarly $\norm{\Upsilon f}_{I, s} = \norm{f}_{I, s}$ so $\Upsilon$ is an $I$-norm isometry. Thus $\Upsilon$ is bounded and hence extends to a $*$-homomorphism from $\Ccal^*(\Sigma; G; \chi) \to \Ccal^*(\Sigma'; G; \chi')$. On the other hand, $\Upsilon^{-1}:\Ccal_c(\Sigma'; G; \chi') \to \Ccal_c(\Sigma; G; \chi)$ is also bounded and extends to a $*$-homomorphism between the $C^*$-algebras.
    It follows that $\Upsilon$ is also an isometry with respect to the full norm and hence, $\Upsilon$ extends to a $*$-isomorphism $\Ccal^*(\Sigma; G; \chi) \cong \Ccal^*(\Sigma'; G; \chi')$. 
    
To show isomorphism of the reduced $C^*$-algebras, we claim that $\Upsilon$ is also an isometry with respect to the $\Ltwo$ norms. To see this, take any $u \in \Sigma^{(0)}$ and $f \in \Ccal_c(\Sigma; G; \chi)$.  We see from Lemma \ref{lem: equivariant Lone norm} that
\begin{equation*}
\begin{split}
    \norm{f}^2_{\Ltwo(\Sigma_u; G_u; \chi)} &= \int_G \abs{f(\sfrak(\gamma))}^2 d\lambda_u(\gamma)\\
    &= \int_G \abs{f'((\theta_1\circ\sfrak)(\sigma))}^2 d\lambda_u(\gamma)\\
    &= \norm{\Upsilon f}_{\Ltwo((\Sigma')_u; G_u; \chi')}^2.
\end{split}
\end{equation*}
 It follows by Lemma \ref{lem: Ltwo density of equivariant Cc} that $\Upsilon$ extends to a unitary operator $\upsilon:\Ltwo(\Sigma_u; G_u; \chi)\to \Ltwo((\Sigma')_u; G_u; \chi')$. For $f,g \in \Ccal_c(\Sigma; G; \chi)$, from $\Upsilon(f*g)=\Upsilon f*\Upsilon g$, proven above, we obtain \[
\Upsilon[\Ind\delta_u(f)(g)]=\Ind\delta_u(\Upsilon f)(\Upsilon g).
\]
This shows that the two $\Ind\delta_u$, for $\Sigma$ and for $\Sigma'$, are unitarily equivalent (with respect to the unitary $\upsilon$). Thus $\Upsilon$ is an isometry with respect to the reduced norm, and extends to a $*$-isomorphism $C_r^*(\Sigma; G; \chi) \cong C_r^*(\Sigma'; G; \chi')$.
\end{proof}

The major consequence of this theorem is that for every locally compact Hausdorff abelian group $\Gamma$, the $\Gamma$-twisted groupoid $C^*$-algebras can be realised as twisted groupoid $C^*$-algebras:

\begin{corollary}
Suppose that $\Gamma$ is compact and that $\chi$ is a character on $\Gamma$. Then the full and reduced $\Gamma$-twisted groupoid $C^*$-algebras associated with $(\Sigma,\iota,\pi)$ and $\chi$ are canonically $*$-isomorphic, respectively, to the full and reduced twisted groupoid $C^*$-algebras associated with a twisted groupoid over $G$. 
\end{corollary}
\begin{proof}
First, apply Theorem \ref{thm: compatible twists construction} with $\Gamma'=\T$, and then apply Theorem \ref{thm: compatible twists have isomorphic algebras} with the $\Gamma'$-twist $(\Sigma',\iota',\pi')$ obtained and with, the character $\chi'=\id_\T$ of $\Gamma'$.
\end{proof}

A further interesting consequence is that, a $\Gamma$-twisted groupoid $C^*$-algebra corresponding to any non-surjective character can be realised from a twist over $G$ by a \textit{finite} subgroup of the unit circle: 

\begin{corollary}
Suppose that $\Gamma$ is compact and that $\chi$ is a character on $\Gamma$ such that $\chi(\Gamma)\neq \T$. Then, there exists a finite subgroup $\Gamma'$ of $\T$ such that the full and reduced $\Gamma$-twisted groupoid $C^*$-algebras associated with $(\Sigma,\iota,\pi)$ and $\chi$ are canonically $*$-isomorphic, respectively, to the full and reduced $\Gamma'$-twisted groupoid $C^*$-algebras associated with a $\Gamma'$-twist over $G$ and the character the inclusion $\Gamma'\xhookrightarrow{}\T$. \qed
\end{corollary}

\printbibliography

\end{document}